\pdfoutput=1

\documentclass[11pt, oneside]{article}   	
\usepackage{graphicx}				
\usepackage{geometry}                		
\usepackage[parfill]{parskip}    		
\usepackage{authblk}

\usepackage{amssymb}
\usepackage{amsmath}
\usepackage{amsthm}                      
\usepackage{enumerate}
\usepackage{tikz-cd}
\usepackage[mode=buildnew]{standalone}
\usepackage{tikz-3dplot}
\usepackage{tikz-cd}
\usepackage{mathtools}
\usepackage{bm}
\usepackage{hyperref}
\usepackage[margin=1.4cm]{caption}

\newcommand{\R}{\mathbb{R}}

\newcommand{\ba}{\mathbf{a}}

\newcommand{\bc}{\mathbf{c}}

\newcommand{\bx}{\mathbf{x}}
\newcommand{\bu}{\mathbf{u}}
\newcommand{\bv}{\mathbf{v}}
\newcommand{\bw}{\mathbf{w}}

\newcommand{\by}{\mathbf{y}}

\newcommand{\bz}{\mathbf{0}}

\newcommand{\im}{\textup{im}}

\newcommand{\Span}{\textup{Span}}
\newcommand{\interior}{\textup{int}}

\newcommand{\norm}[1]{\left\lVert#1\right\rVert}

\newtheorem{theorem}{Theorem}[section]
\newtheorem{lemma}[theorem]{Lemma}
\newtheorem{proposition}[theorem]{Proposition}

\newtheorem{definition}[theorem]{Definition}

\newtheorem{remark}[theorem]{Remark}
\newtheorem{prop_def}[theorem]{Proposition/Definition}

\newcommand{\Addresses}{{
  \bigskip
  \footnotesize

  J.~Love, \textsc{San Francisco, California}\par\nopagebreak
  \textit{E-mail address}: \texttt{jack.eddie.love@gmail.com}

}}

\title{The Stable Limit of Moduli Spaces of Polygons}
\author{Jack Love}
\date{\today}

\begin{document}
\maketitle



\abstract{
Polygon spaces have been studied extensively (\cite{millson}, \cite{Hausmann1996PolygonSA}, \cite{millson2}, \cite{grassmannians}, \cite{signature}, \cite{history}, \cite{mandini}, \cite{manon}, \cite{farber}, \cite{higgs}, \cite{kourganoff}, \cite{Cantarella2016TheSG}), and yet missing from the literature is a simple property that every polygon has: dimension. This is distinct (possibly) from the dimension of the ambient space in which the polygon lives. A square, in the usual sense of the word, is $2$-dimensional no matter the dimension of the ambient space in which it is embedded. If the ambient space has dimension greater than or equal to $3$ we may bend the square along a diagonal to produce a $3$-dimensional polygon with the same edge-lengths. And yet even if the dimension of the ambient space is large, no amount of bending of the square will produce a polygon of dimension larger than $3$. We generalize this idea to show that there are only finitely many moduli spaces of polygons with given edge-lengths, even as the ambient dimension increases without bound.
}

\section{Introduction}

A {\em polygon} is a chain of line segments, called edges, that begins and ends at the origin in some Euclidean space $\R^d$. A {\em polygon space} is the collection of all polygons in $\R^d$ with a given vector of edge-lengths $\ell=(l_1, \ldots, l_n)$. A {\em moduli space of polygons} is the quotient of a polygon space by rotations in $\R^d$. For context and a point of contrast, we mention the work of Michael Farber and Viktor Fromm in \cite{farber}. It is well-known that polygon spaces are smooth manifolds if $\ell$ is sufficiently generic. Farber and Fromm showed that upon fixing the dimension $d$ of the ambient space and allowing $\ell$ to vary generically, the diffeomorphism types of polygon spaces are in bijection with the components of a discrete geometric object. Rather than looking at diffeomorphism types of polygon spaces with fixed $d$ and varying $\ell$, we look at homeomorphism types of moduli spaces of polygons with fixed $\ell$ and varying $d$. Since any polygon in $\R^d$ can be embedded in $\R^{d+1}$, we define a directed system of moduli spaces of polygons by fixing $\ell=(l_1, \ldots, l_n)$ and letting $d$ increase without bound. Our main theorem is presented in Section \ref{sec:stable}, but we state a version of it here:

\begin{theorem}\label{thm_intro}
Given a vector $\ell=(l_1, \ldots, l_n)$ of edge-lengths, the directed system of moduli spaces of $\ell$-gons stabilizes when the ambient dimension is equal to $n$.
\end{theorem}

At the crux of this result is the notion of polygon dimension. We show that the dimension of a polygon with $n$ edges is bounded by a function of $n$, and also plays a role in the relationship between reflections and rotations of the polygon. These intermediate results give conditions for when the arrows in the directed system mentioned above are either injective or surjective and ultimately tell us that they are homeomorphisms when $d\geq n$. Taken together with the result of Farber and Fromm, this implies that upon fixing the number $n$ of edges, there are only finitely many homeomorphism types of polygon spaces even as $\ell$ varies generically throughout $\R^n$ and $d$ goes to infinity.

In Section \ref{sec:definitions} we define the directed system of moduli spaces of polygons and introduce polygon dimension. In Section \ref{sec:dimension} we look at the possible dimensions for polygons in a given polygon space, and Section \ref{sec:orbits} looks at the interplay between polygon dimension, ambient space dimension, and orbits of polygons under rotations and reflections. The results of these sections will allow us to state and prove our main theorem in Section \ref{sec:stable}. Lastly, in Section \ref{sec:example} we illustrate our result with an example and use the example to motivate new questions.

\begin{figure}
\centering
\includestandalone[scale=1]{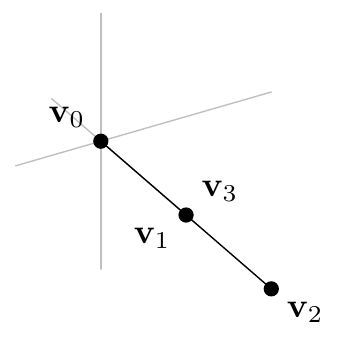}\qquad
\includestandalone[scale=1]{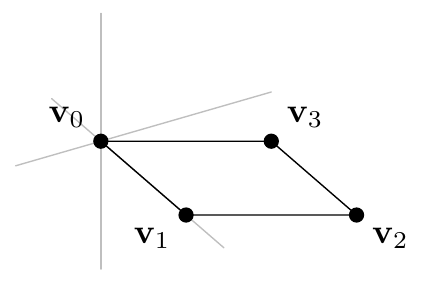}\qquad
\includestandalone[scale=1]{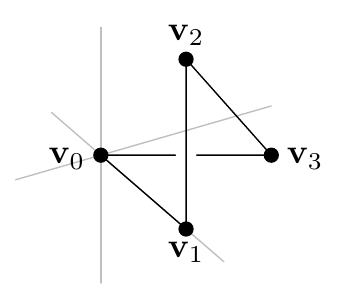}
\caption{Pictures of $1$-, $2$-, and $3$-dimensional $\ell$-gons in $\R^3$, for $\ell=(1,1,1,1)$.} 
\label{fig:polygons}
\end{figure}

\subsection{Acknowledgments}\label{subsec:acknowledgments}

This work is based on the author's Ph.D. thesis at George Mason University, begun under Chris Manon and directed by Sean Lawton. The author is grateful to them both for their guidance, as well as to Jim Lawrence, Neil Epstein, and Rebecca Goldin for their helpful comments. The author also acknowledges the work of John Millson and Michael Kapovich, and Michael Farber and Viktor Fromm, whose papers provided much of the background and context for what is presented here.

\section{Definitions}\label{sec:definitions}

The goal of this section is to define the directed system of moduli spaces of $\ell$-gons, and introduce polygon dimension which will be our main tool in proving that these systems stabilize. We begin with the definition of polygon space.

\begin{definition}\label{def:vertexDescription}
Given $n\geq 3$, $\ell=(l_1, \ldots, l_n)\in\R_{\scriptscriptstyle{>0}}^n$, and $d\geq 2$, the {\em polygon space} $V_d(\ell)$ is the topological subspace of $\R^{d(n-1)}$ defined as
\[V_d(\ell)=\left\{P=(\bv_1, \ldots, \bv_{n-1})\in \R^{d(n-1)} \colon \norm{ \bv_i-\bv_{i-1} }=l_i,\, i=1, \ldots, n\right\},\]
where $\bv_0=\bv_n=\bz$. Elements of $V_d(\ell)$ are called {\em polygons} or $\ell${\em -gons}. Given an $(l_1, \ldots, l_n)$-gon $(\bv_1, \ldots, \bv_{n-1})$, the $\bv_i$ are its {\em vertices} and the $l_i$ are its {\em edge-lengths}.
\end{definition}

The polygon space $V_d(\ell)$ admits natural $SO(d)$ and $O(d)$ actions by the map
\[V_d(\ell)\times G\to V_d(\ell), \quad ((\bv_1, \ldots, \bv_{n-1}),T) \mapsto (T(\bv_1), \ldots, T(\bv_{n-1})),\]
where $G$ is $SO(d)$ or $O(d)$. We write $T(P)$ to mean the image of $(P, T)$ under this map and we let $G(P)=\{T(P) \colon T\in G\}$ denote the orbit of $P$ under the action of $G$. Thus two polygons in $V_d(\ell)$ belong to the same $SO(d)$ orbit if they are rotations of one another, and they belong to the same $O(d)$ orbit if they are rotations or reflections of one another. 

\begin{definition}
The {\em moduli space of $\ell$-gons in $\R^d$}, denoted $M_d(\ell)$, is the quotient space $V_d(\ell)/SO(d)$. We let $\pi:V_d(\ell)\to M_d(\ell)$ denote the canonical projection, and given $P=(\bv_1, \ldots, \bv_{n-1})\in V_d(\ell)$ we let $[P]=[\bv_1, \ldots, \bv_{n-1}]$ denote $\pi(P)$.
\end{definition}

Since any $\ell$-gon in $\R^d$ can be thought of as an $\ell$-gon in $\R^{d+1}$ by embedding it in a codimension-$1$ hyperplane, there is a natural map $M_d(\ell)\to M_{d+1}(\ell)$. In this way, we are led to a directed system of moduli spaces of $\ell$-gons for fixed $\ell$. We make this concrete in the following proposition/definition.

\begin{prop_def}\label{prop:direct_system}
Let $\R^2\xrightarrow{f_2}\R^3\xrightarrow{f_3}\R^4\xrightarrow{f_4}\cdots$ be any directed system of linear Euclidean isometries. Fix $\ell\in\R_{\scriptscriptstyle{>0}}^n$ and for $d\geq 2$, define the maps
\begin{align*}
F_d:V_d(\ell) &\to V_{d+1}(\ell) \\
(\bv_1, \ldots \bv_{n-1}) &\mapsto (f_d(\bv_1), \ldots, f_d(\bv_{n-1}))
\end{align*}
and
\begin{align*}
\varphi_d:M_d(\ell)&\to M_{d+1}(\ell) \\
[P] &\mapsto [F_d(P)].
\end{align*}
Then $\varphi_d$ is continuous and the directed system $\langle M_d(\ell), \varphi_d \rangle_{d\geq 2}$ is independent of the choice of maps $f_d$. Thus given $\ell\in\R_{\scriptscriptstyle{>0}}^n$ and any directed system $\langle \R^d, f_d\rangle_{d\geq 2}$ of linear Euclidean isometries, we call $\langle M_d(\ell), \varphi_d \rangle_{d\geq 2}$ the {\em directed system of moduli spaces of $\ell$-gons}.
\end{prop_def}

\begin{proof}
Let $\ell\in\R_{\scriptscriptstyle{>0}}^n$, fix a directed system $\langle \R^d, f_d\rangle_{d\geq 2}$ of linear Euclidean isometries, and consider the diagram
\[
\begin{tikzcd}
V_d(\ell) \arrow[r, "F_d"] \arrow[d, "\pi_d"] & V_{d+1}(\ell) \arrow[d, "\pi_{d+1}"] \\
M_d(\ell) \arrow{r}{\varphi_d} & M_{d+1}(\ell)
\end{tikzcd}
\]
where $\pi_k:V_k(\ell)\to M_k(\ell)$ is the canonical projection map. As with any topological transformation group, these projection maps are open and continuous. The map $F_d$ is continuous as well, and thus if $U\subset M_{d+1}(\ell)$ is open, then $\pi_d(F_d^{-1}(\pi_{d+1}^{-1}(U)))$ is open. But this is precisely $\varphi_d^{-1}(U)$, and thus $\varphi_d$ is continuous.

To show that $\langle M_d(\ell), \varphi_d \rangle_{d\geq 2}$ is independent of our choice of $f_d$'s, let $\langle \R^d, f_d\rangle_{d\geq 2}$ and $\langle \R^d, g_d\rangle_{d\geq 2}$ be two directed systems of linear Euclidean isometries and define the maps
\begin{align*}
F_d:V_d(\ell) &\to V_{d+1}(\ell) \\
(\bv_1, \ldots \bv_{n-1}) &\mapsto (f_d(\bv_1), \ldots, f_d(\bv_{n-1}))
\end{align*}
and
\begin{align*}
G_d:V_d(\ell) &\to V_{d+1}(\ell) \\
(\bv_1, \ldots \bv_{n-1}) &\mapsto (g_d(\bv_1), \ldots, g_d(\bv_{n-1})).
\end{align*}
Since $SO(d+1)$ acts transitively on orthonormal bases of proper linear subspaces of $\R^{d+1}$, for every $d\geq 2$ and every $P=(\bv_1, \ldots \bv_{n-1})\in V_d(\ell)$ there exists $T_d\in SO(d+1)$ such that $g_d(\bv_i)=T_d(f_d(\bv_i))$, and thus $[G_d(P)]=[F_d(P)]=\varphi_d([P])$.
\end{proof}

At each step $M_d(\ell)\xrightarrow{\varphi_d} M_{d+1}(\ell)$ in this system, potentially two things are happening. First, there are new $\ell$-gons appearing in $M_{d+1}(\ell)$ that are not in the image of $\varphi_d$, in which case $\varphi_d$ is not surjective; second, polygons $P$ and $Q$ that were not identified in $M_d(\ell)$ become identified in $M_{d+1}(\ell)$, in which case $\varphi_d$ is not injective. Our main result is that neither of these possibilities occurs when, and only when, $d\geq n$, and thus the system stabilizes there. All of these results rely on the notion of polygon dimension, and we conclude this section with its definition.

\begin{definition}\label{def:dimension}
Let $P=(\bv_1, \ldots, \bv_{n-1})\in V_d(\ell)$. The {\em dimension} of $P$, denoted $\dim(P)$, is the dimension of $\Span(\{\bv_1, \ldots, \bv_{n-1}\})$ as a linear subspace of $\R^d$, and we say that $P$ is a $\dim(P)$-dimensional polygon. Given $[P]\in M_d(\ell)$ the dimension of $[P]$ is the dimension of $P$.
\end{definition}

\section{Dimensions of polygons in $V_d(\ell)$}\label{sec:dimension}

In this section, we show that, as long as $\ell$ is not too degenerate, the polygon space $V_d(\ell)$ contains polygons of every dimension from $2$ up through an upper bound dependent on $d$ and $\ell$. Let us first consider the conditions on a vector $\ell\in\R_{\scriptscriptstyle{>0}}^n$ for which the polygon space $V_d(\ell)$ is nonempty. Just as the triangle inequalities give necessary and sufficient conditions for three positive real numbers to be the edge-lengths of a triangle, the polygon inequalities
\begin{equation}
\label{eqn:ineqs}
l_i\leq\sum_{j\neq i}l_j, \quad l_i>0 \qquad i=1, \ldots, n
\end{equation}
give necessary and sufficient conditions for $(l_1,\ldots, l_n)$ to be the edge-lengths of a polygon. (See Lemma 1 in \cite{millson}, keeping in mind that their polygons have been normalized so that $\sum_{i=1}^n l_i = 1$ and edges of length 0 are allowed.) In other words, letting $C_n$ denote the solution space of the inequalities \eqref{eqn:ineqs}, the polygon space $V_d(\ell)$ is nonempty if and only if $\ell\in C_n$. It may be helpful to think of $C_n$ as ``almost a polyhedron'', in that the topological closure
\[\overline{C_n}=\left\{\ell\in\R^n\,\middle|\, l_i\leq\sum_{j\neq i}l_j, \quad l_i\geq 0 \qquad i=1, \ldots, n\right\}\]
is a polyhedral cone pointed at the origin. Now we give an upper bound on the dimension of polygons in $V_d(\ell)$ that holds for all $\ell\in C_n$.

\begin{proposition}\label{prop:dimension}
Let $\ell\in C_n$ and let $P\in V_d(\ell)$. Then $\dim(P)\leq\min\{n-1,d\}$.
\end{proposition}

\begin{proof}
This is immediate since $\Span(\bv_1, \ldots, \bv_{n-1})$ is a linear subspace of $\R^d$ spanned by $n-1$ vectors.
\end{proof}

This simple observation---that a function of $\ell$ bounds the dimension of $\ell$-gons---is central to our main theorem. Together with the upcoming results, it implies both the injectivity and surjectivity of the maps $\varphi_d$ for large $d$. Before moving forward we characterize a subset of $C_n$ for which the corresponding moduli spaces are rather trivial. The border of $C_n$---the intersection of $C_n$ with its boundary---consists of those $\ell$ such that $l_i=\sum_{j\neq i}l_j$ for some $i=1, \ldots, n$. If $\ell$ lies on the border of $C_n$ then $V_d(\ell)$ consists only of 1-dimensional polygons, since the condition $l_i=\sum_{j\neq i}l_j$ forces $\bv_i$ and $\bv_{i-1}$ to be the endpoints of a line segment containing the rest of the vertices. In this case, $SO(d)$ acts transitively on $V_d(\ell)$ so $M_d(\ell)$ is a singleton. Thus from here forward, we will assume $\ell\in\interior(C_n)$.

It is well-known that if $\ell\in\interior(C_n)$ then $V_d(\ell)$ contains $2$-dimensional polygons. Proposition \ref{prop:dimensionful} below says that $V_d(\ell)$ contains $k$-dimensional polygons for all values of $k$ from $2$ to $\min\{d, n-1\}$. In the proof we construct a $(k+1)$-dimensional $\ell$-gon from a $k$-dimensional $\ell$-gon for any $2\leq k<\min\{d, n-1\}$. In the construction, we use a vertex with a special property, and Lemma \ref{lem:vertex_handle} shows that such a vertex exists. Refer to Figures \ref{fig:vertex_handle} and \ref{fig:folding} when reading Lemma \ref{lem:vertex_handle} and Proposition \ref{prop:dimensionful}, respectively.

\begin{lemma}\label{lem:vertex_handle}
Let $P=(\bv_1, \ldots, \bv_{n-1})\in V_d(\ell)$ with $2\leq \dim(P)< n-1$. For each $i=1, \ldots, n-1$, let $L_i$ denote the line through $\bv_{i-1}$ and $\bv_{i+1}$, and let $U_i$ denote $\Span(\{\bv_j\}_{j\neq i})$. There exists $i\in\{1, \ldots, n\}$ such that $\bv_i\in U_i \setminus L_i$ .
\end{lemma}

\begin{proof}
Let $P\in V_d(\ell)$ with $2\leq \dim(P)<n-1$. Since $\dim(P)<n-1$ there exists a linear dependence among $\{\bv_1, \ldots, \bv_{n-1}\}$, and thus $\bv_{t}\in U_{t}$ for some $t=1, \ldots, n-1$. If $\bv_{t}\not\in L_{t}$ we are done by taking $i=t$, so suppose $\bv_{t}\in L_{t}$. Then $\bv_{t-1}, \bv_t, \bv_{t+1}$ are colinear and in particular $\bv_{t+1}\in\Span(\{\bv_{t-1}, \bv_t\})$, so $\bv_{t+1}\in U_{t+1}$. Now if $\bv_{t+1}\not\in L_{t+1}$ we are done by taking $i=t+1$ so suppose $\bv_{t+1}\in L_{t+1}$. Then by the same argument as above $\bv_{t+2}\in U_{t+2}$, and we also have $\bv_{t-1}, \bv_t, \bv_{t+1}, \bv_{t+2}$ colinear. Continuing in this way, if $\bv_i\not\in U_i\setminus L_i$ for any $i \in \{1, \ldots, n-1\}$, then all of $\bv_0, \bv_1, \ldots, \bv_{n-1}$ are co-linear, contradicting our assumption that $\dim(P)\geq2$.
\end{proof}

\begin{figure}[h]
\centering
\includestandalone[scale=1]{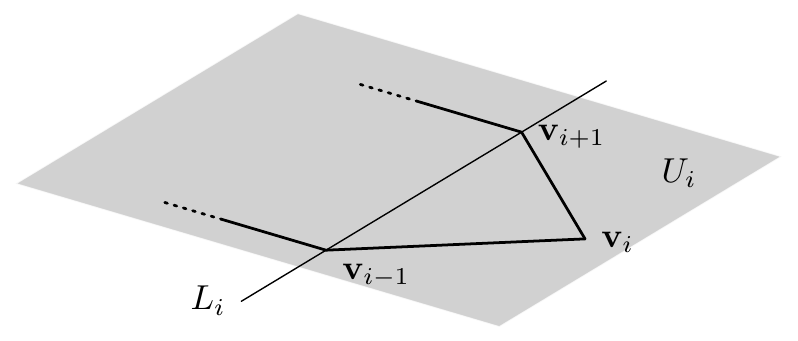}
\caption[]{If the dimension of a polygon is small enough, there exists a vertex ($\bv_i$ in the figure) that lies in the span of the other vertices and is not co-linear with its neighbors.} 
\label{fig:vertex_handle}
\end{figure}

\begin{proposition}\label{prop:dimensionful}
If $\ell\in\interior(C_n)$, then $V_d(\ell)$ contains $k$-dimensional polygons for all $k\in\{2, \ldots, \min\{d, n-1\}\}$.
\end{proposition}

\begin{proof}
Let $\ell\in\interior(C_n)$ and let $P=(\bv_1, \ldots, \bv_{n-1})\in V_d(\ell)$ with $2\leq \dim(P)<\min\{d, n-1\}$. For $i=1, \ldots, n-1$ define $L_i$ and $U_i$ as in Lemma \ref{lem:vertex_handle}, and choose $i$ so that $\bv_i\in U_i \setminus L_i$. Let $\bc$ be the orthogonal projection of $\bv_i$ onto $L_i$ and let $\bu\not\in U_i$ be a unit vector. Define $\bw:=\bc+\norm{ \bv_i-\bc } \bu$ and $Q=(\bv_1, \ldots, \bv_{i-1}, \bw, \bv_{i+1}, \ldots, \bv_{n-1})$. Since $\bv_{i-1},\, \bv_i, \,\bv_{i+1}$ are not colinear, $\bv_i\neq\bc$ and thus $\bw\not\in U_i$, so $\dim(Q)=\dim(P)+1$. It remains to show that $\norm{ \bw - \bv_{i-1} }=l_{i-1}$ and $\norm{ \bw - \bv_{i+1} }= l_{i+1}$ so that $Q\in V_d(\ell)$. Let $\by$ be $\bv_{i-1}$ or $\bv_{i+1}$. We have
\begin{align}
\norm{ \bw - \by }^2 &= \norm{ \bw - \bc }^2+ \norm{ \bc -\by } ^2 \label{eqn:pythag1}\\
&= \norm{ \bv_i-\bc }^2 + \norm{ \bc - \by }^2 \label{eqn:rotate}\\
&= \norm{ \bv_i-\by }^2, \label{eqn:pythag2}
\end{align}
where Equations \eqref{eqn:pythag1} and \eqref{eqn:pythag2} are the Pythagorean theorem, and Equation \eqref{eqn:rotate} follows from the definition of $\bw$. The result follows.
\end{proof}

\begin{figure}[h]
\centering
\includestandalone[scale=1]{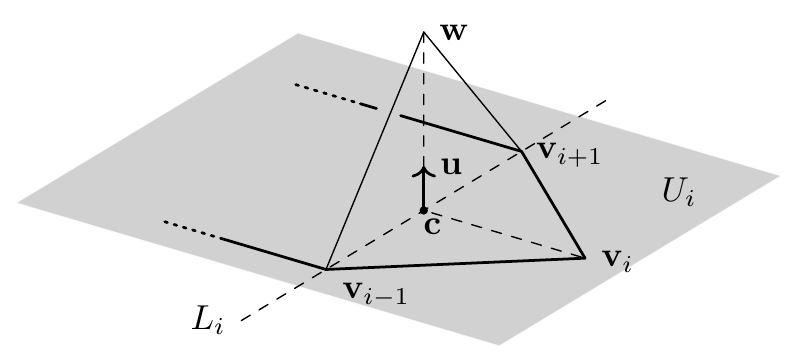}
\caption{The vertex $\bv_i$, guaranteed by Lemma \ref{lem:vertex_handle}, of a $k$-dimensional $\ell$-gon for sufficiently small $k$, allows the $\ell$-gon to bend along the line $L_i$ into a $(k+1)$-dimensional $\ell$-gon, whose vertices are identical with those of the original $\ell$-gon except that $\bw$ has replaced $\bv_i$ (here, $\bu$ is orthogonal to $U_i$ to make the picture easy to read, but the construction above does not require this).} 
\label{fig:folding}
\end{figure}

In the following remark, we observe two implications of the results in this section on the directed system $\langle M_d(\ell), \varphi_d\rangle$---we will make these observations concrete in Section \ref{sec:stable}, Proposition \ref{prop:surjective}.

\begin{remark}\label{rem:surjective}
The map $\varphi_d$ is not surjective if $d<n-1$ since new polygons of higher dimensions continue to appear until we reach $M_{n-1}(\ell)$, and $\varphi_d$ is surjective for $d=n-1$ (and thereafter) since every polygon in $M_n(\ell)$ has dimension less than $n$ and is thus the image of a polygon in $M_{n-1}(\ell)$.
\end{remark}

Given these observations about the surjectivity of $\varphi_d$, it remains to make similar observations about the injectivity of $\varphi_d$. We make these observations in the following section.

\section{Orbits in the directed system $\langle M_d(\ell), \varphi_d \rangle$}\label{sec:orbits}

We begin with some preliminary remarks about $O(d)$ and $SO(d)$ and their actions on polygons. The orthogonal group $O(d)$ has two connected components: one is $SO(d)$ and the other, consisting of reflections through codimension-1 hyperplanes, we will call $SO(d)^-$. Though $SO(d)^-$ is not a group we may none-the-less consider its action on $V_d(\ell)$ by defining $SO(d)^-(P)=\{T(P) \colon T\in SO(d)^-\}$ for every $P\in V_d(\ell)$. Thus for any $P\in V_d(\ell)$ we have
\begin{equation}\label{eqn:partition1
}O(d)(P)=SO(d)(P)\cup SO(d)^-(P).
\end{equation}
Moreover, $SO(d)$ acts transitively by left (or right) multiplication on $SO(d)^-$, so that given any $T\in SO(d)^-$ we may recover all of $SO(d)^-$ by composing $T$ with elements of $SO(d)$. Thus given any $P\in V_d(\ell)$ and any $T\in SO(d)^-$ we have
\begin{equation}\label{eqn:partition2}
O(d)(P)=SO(d)(P)\cup SO(d)(T(P)).
\end{equation}
We also note that $SO(d)$ acts transitively on orthonormal bases of $k$-dimensional linear subspaces of $\R^d$ if $k<d$. Thus if $X$ and $Y$ are proper linear subspaces of $\R^d$ of dimension $k$, there exists $T\in SO(d)$ so that $T(X)=Y$. Our final preliminary is the following lemma.

\begin{lemma}\label{lem:rotating_subspaces}
If $X$ and $Y$ are proper linear subspaces of $\R^d$ with $X\subset Y$ and $T\in SO(d)$ is such that $T(X)\subset Y$, then there exists $S\in SO(d)$ such that $S(Y)=Y$ and $S|_X=T|_X$.
\end{lemma}

\begin{proof}
Let $X$, $Y$, and $T$ be as above. Let $\{\bu_1, \ldots, \bu_k\}$ be an orthonormal basis for $X$. Then $\{T(\bu_1), \ldots, T(\bu_k)\}$ is an orthonormal basis for $T(X)$. If $X$ and $T(X)$ are both subspaces of $Y$ we may extend these to orthonormal bases $\{\bu_1, \ldots, \bu_k, \bv_{k+1} \ldots, \bv_l\}$ and $\{T(\bu_1), \ldots, T(\bu_k), \bw_{k+1}, \ldots, \bw_l\}$ for $Y$. Since $Y$ is a proper linear subspace of $\R^d$ there exists $S\in SO(d)$ such that $S(\bu_i)=T(\bu_i)$ for $i=1, \ldots, k$ and $S(\bv_i)=\bw_i$ for $i=k+1, \ldots, l$. Thus $S(Y)=Y$ and $S\vert_X=T\vert_X$.
\end{proof}

The main results of this section are two interpretations of the informal statement ``reflections in $\R^d$ are rotations in $\R^{d+1}$''. One interpretation, Proposition \ref{prop:polygon_orbit_embedding}, is that if $P$ and $Q$ belong to the same $SO(d+1)$ orbit in $V_{d+1}(\ell)$, then they come from polygons in the same $O(d)$ orbit of $V_d(\ell)$. The second interpretation, Proposition \ref{prop:equal_orbits}, is that if the dimension of a polygon is less than the dimension of the ambient space, then every reflection of that polygon may be obtained by rotation. We first state two lemmas that lead to Proposition \ref{prop:polygon_orbit_embedding}.

\begin{lemma}\label{lem:orbit_embedding}
Let $f_d:\R^d\to\R^{d+1}$ be a linear Euclidean isometry. For every $T\in O(d)$ there exists $S\in SO(d+1)$ such that $S(f_d(\bx))=f_d(T(\bx))$ for all $\bx\in\R^d$. Conversely, for every $S\in SO(d+1)$ such that $S(\im(f_d))=\im(f_d)$, there exists $T\in O(d)$ such that $S(f_d(\bx))=f_d(T(\bx))$ for all $\bx\in\R^d$.
\end{lemma}

\begin{proof}
Let $\{\ba_1, \ldots, \ba_d\}$ be a basis for $\R^d$ and let $f_d:\R^d\to\R^{d+1}$ be a linear Euclidean isometry. There is a natural embedding
\begin{align*}
\epsilon:O(d)&\to SO(d+1) \\
T &\mapsto\begin{cases}
T\oplus\bm{1} &\text{ if }T\in SO(d) \\
T\oplus -\bm{1} &\text{ if }T\in SO(d)^-
\end{cases}
\end{align*}
where $\bm{1}$ is the $1\times 1$ identity matrix, and the matrix $\epsilon(T)$ is written with respect to a basis for $\R^{d+1}$ whose first $d$ elements are $\{f_d(\ba_1), \ldots, f_d(\ba_d)\}$. Thus if $T\in O(d)$ then $\epsilon(T)\in SO(d+1)$ and $\epsilon(T)(f_d(\bx))=f_d(T(\bx))$ for all $\bx\in\R^d$. Conversely, let $S\in SO(d+1)$ such that $S(\im(f_d))=\im(f_d)$. Then $S\in\im(\epsilon)$ and we have $S(f_d(\bx))=f_d(\epsilon^{-1}(S)(\bx))$ for all $\bx\in\R^d$.
\end{proof}

Lemma \ref{lem:orbit_embedding} says that rotating a point in the image of $f_d$ is the same as rotating or reflecting its preimage in $\R^d$ and then pushing it forward. Lemma \ref{lem:polygon_orbit_embedding} extends this idea from points to polygons.

\begin{lemma}\label{lem:polygon_orbit_embedding}
Given $P\in V_d(\ell)$ and $T\in O(d)$, there exists $S\in SO(d+1)$ such that $S(F_d(P))=F_d(T(P))$. Conversely, given $F_d(P)\in\im(F_d)$ and $S\in SO(d+1)$ such that $S(F_d(P))\in\im(F_d)$, there exists $T\in O(d)$ such that $S(F_d(P))=F_d(T(P))$.
\end{lemma}

\begin{proof}
Let $P=(\bv_1, \ldots, \bv_{n-1})\in V_d(\ell)$ and $T\in O(d)$. By Lemma \ref{lem:orbit_embedding} there exists $S\in SO(d+1)$ such that $S(f_d(\bv_i))=f_d(T(\bv_i))$ for all $i=1, \ldots, n-1$, and thus $S(F_d(P))=F_d(T(P))$.

Conversely, let $F_d(P)\in \im(F_d)$ and $S\in SO(d+1)$ such that $S(F_d(P))\in \im(F_d)$. Let $X=\Span(\{f_d(\bv_1), \ldots, f_d(\bv_{n-1})\})$ and $Y=\im(f_d)$. Then $X$ and $Y$ are proper linear subspaces of $\R^{d+1}$ and $X$ and $S(X)$ are subsets of $Y$. Thus by Lemma \ref{lem:rotating_subspaces} there exists $S'\in SO(d+1)$ such that $S'(Y)=Y$ and $S'\vert_X=S\vert_X$. Since
\[S'(\im(f_d))=\im(f_d),\] Lemma \ref{lem:orbit_embedding} says there exists $T\in O(d)$ such that $S'(f_d(\bx))=f_d(T(\bx))$ for all $\bx\in\R^d$, and thus $S'(F_d(P))=F_d(T(P))$. Since
\[S'|_{\Span(\{f_d(\bv_1), \ldots, f_d(\bv_{n-1})\})}=S\vert_{\Span(\{f_d(\bv_1), \ldots, f_d(\bv_{n-1})\})},\] we have $S(F_d(P))=F_d(T(P))$.

\end{proof}

\begin{proposition}\label{prop:polygon_orbit_embedding}
Let $P\in V_d(\ell)$. The preimage under $F_d$ of the $SO(d+1)$ orbit of $F_d(P)$ is the $O(d)$ orbit of $P$:
\[F_d^{-1}(SO(d+1)(F_d(P)))=O(d)(P).\]
\end{proposition}

\begin{proof}
Let $P\in V_d(\ell)$. We begin with the inclusion $F_d^{-1}(SO(d+1)(F_d(P)))\subset O(d)(P)$. Let $Q\in F_d^{-1}(SO(d+1)(F_d(P)))$. Then $F_d(Q)=S(F_d(P))$ for some $S\in SO(d+1)$. Since $F_d(P)$ and $S(F_d(P))$ are both in the image of $F_d$, by Lemma \ref{lem:polygon_orbit_embedding} there exists $T\in O(d)$ such that $F_d(T(P))=S(F_d(P))$, and thus $F_d(T(P))=F_d(Q)$. Since $F_d$ is injective, $T(P)=Q$, and so $Q\in O(d)(P)$.

For the inclusion $ O(d)(P)\subset F_d^{-1}(SO(d+1)(F_d(P)))$, let $Q\in O(d)(P)$ and let $T\in O(d)$ such that $Q=T(P)$. By Lemma \ref{lem:polygon_orbit_embedding} there exists $S\in SO(d+1)$ such that $S(F_d(P))=F_d(T(P))$. Thus $S(F_d(P))=F_d(Q)$, so $Q\in F_d^{-1}(SO(d+1)(F_d(P)))$.
\end{proof}

In the context of the map $\varphi_d:M_d(\ell)\to M_{d+1}(\ell)$, Proposition \ref{prop:polygon_orbit_embedding} says that $\varphi_d$ is not a priori injective. It says the preimage of $[F_d(P)]\in M_{d+1}(\ell)$ in $M_d(\ell)$ is the $\pi_d$ projection of the $O(d)$ orbit of $P$, which may be strictly larger than its $SO(d)$ orbit. However, as the next proposition shows, this is only the case if the dimension of $P$ matches the dimension $d$ of the ambient space.

\begin{proposition}\label{prop:equal_orbits}
Let $P\in V_d(\ell)$. Then $O(d)(P)=SO(d)(P)$ if and only if $\dim(P)<d$.
\end{proposition}

\begin{proof}
Let $P=(\bv_1, \ldots, \bv_{n-1})\in V_d(\ell)$. Given the decomposition
\[O(d)(P)=SO(d)(P)\cup SO(d)^-(P)\]
it suffices to show that $SO(d)^-(P)\subset SO(d)(P)$ if and only if $\dim(P)<d$. Suppose $\dim(P)<d$. Let $Q\in SO(d)^-(P)$, and let $T\in SO(d)^-$ such that $T(P)=Q$. Let $B_1=\{\bu_1, \ldots, \bu_{\dim(P)}\}$ be an orthonormal basis for $\Span(\{\bv_1, \ldots, \bv_{n-1}\})$. Then $B_2=\{T(\bu_1), \ldots, T(\bu_{\dim(P)})\}$ is an orthonormal basis for $\Span(\{T(\bv_1), \ldots, T(\bv_{n-1})\})$. Since $\dim(P)<d$, $B_1$ and $B_2$ are orthonormal bases of proper subspaces of $\R^d$, so there exists $S\in SO(d)$ such that $S(\bu_i)=T(\bu_i)$ for all $i=1, \ldots, \dim(P)$. Thus $S(\bv_i)=T(\bv_i)$ for all $i=1, \ldots, n-1$, so $S(P)=T(P)$, and thus $Q\in SO(d)(P)$.

Now suppose $\dim(P)=d$, and suppose for a contradiction that $SO(d)^-(P)\subset SO(d)(P)$. Then given $T\in SO(d)^-$ there exists $S\in SO(d)$ such that $T(P)=S(P)$, and thus $T(\bv_i)=S(\bv_i)$ for all $i=1, \ldots, n-1$. But since $\dim(P)=d$, some subset of vertices of $P$ form a basis for $\R^d$, and thus $T=S$, a contradiction.
\end{proof}

\begin{remark}\label{rem:injective}
Propositions $\ref{prop:polygon_orbit_embedding}$ and $\ref{prop:equal_orbits}$ imply that $\varphi_d$ is injective on the set of polygons of dimension less than $d$. In Proposition $\ref{prop:injective}$ in the following section, we state a stronger result and describe explicitly the preimage under $\varphi_d$ of $\varphi_d([P])$ for any $[P]\in M_d(\ell)$.
\end{remark}

\section{Stabilization theorem}\label{sec:stable}

In this section we give our main result, stated as Theorem \ref{thm:stable}, which says that aside from trivial cases the directed system $\langle M_d(\ell), \varphi_d\rangle_{d\geq 2}$ of moduli spaces of $\ell$-gons stabilizes when $d$ is equal to the length of $\ell$. Remarks \ref{rem:surjective} and \ref{rem:injective} summarize our results so far: the map $\varphi_d$ is surjective if and only if $d\geq n-1$, and injective on the set of polygons of dimension less than $d$. We begin by formalizing Remarks \ref{rem:surjective} and \ref{rem:injective} with Propositions \ref{prop:surjective} and \ref{prop:injective}, respectively.

\begin{proposition}\label{prop:surjective}
Let $\ell\in\interior(C_n)$. The map $\varphi_d:M_d(\ell)\to M_{d+1}(\ell)$ is surjective if and only if $d\geq n-1$.
\end{proposition}

\begin{proof}
Let $d< n-1$. By Proposition \ref{prop:dimensionful}, $M_{d+1}(\ell)$ contains $(d+1)$-dimensional polygons. Since $M_d(\ell)$ does not contain $(d+1)$-dimensional polygons, and since $\dim(\varphi_d([P]))=\dim([P])$, $\varphi_d$ is not surjective. Now let $d\geq n-1$, and let $[P]=[\bv_1, \ldots, \bv_{n-1}]\in M_{d+1}(\ell)$. We will show that $[P]\in\im(\varphi_d)$. By Proposition \ref{prop:dimension} we have $\dim(P)<d+1$, so there is a $d$-dimensional linear subspace $X\subset\R^{d+1}$ such that $\bv_i\in X$ for all $i=1, \ldots, n-1$. Let $T\in SO(d+1)$ such that $T(X)=\im(f_d)$. Then $T(\bv_i)\in\im(f_d)$ for all $i=1, \ldots, n-1$, so we may write
\[T(P)=(T(\bv_1), \ldots, T(\bv_{n-1}))=(f_d(\bw_1), \ldots, f_d(\bw_{n-1}))\]
for some $\bw_1, \ldots, \bw_{n-1} \in\R^d$. Since $(f_d(\bw_1), \ldots, f_d(\bw_{n-1}))\in V_{d+1}(\ell)$ and $f_d$ is a linear isometry, $(\bw_1, \ldots, \bw_{n-1})\in V_d(\ell)$. Thus $T(P)=F_d((\bw_1, \ldots, \bw_{n-1}))\in\im(F_d)$ and $[T(P)]\in\im(\varphi_d)$. Finally, since $T\in SO(d+1)$, $[T(P)]=[P]$.
\end{proof}

\begin{proposition}\label{prop:injective}
The map $\varphi_d:M_d(\ell)\to M_{d+1}(\ell)$ is $2$-to-$1$ on the set of $d$-dimensional polygons and $1$-to-$1$ elsewhere.
\end{proposition}

\begin{proof}
Recall the diagram
\[
\begin{tikzcd}
V_d(\ell) \arrow[r, "F_d"] \arrow[d, "\pi_d"] & V_{d+1}(\ell) \arrow[d, "\pi_{d+1}"] \\
M_d(\ell) \arrow{r}{\varphi_d} & M_{d+1}(\ell)
\end{tikzcd}
\]
from Section \ref{sec:definitions}.
Tracing backwards from $M_{d+1}(\ell)$ we have
\[\varphi_d^{-1}\circ\varphi_d([P])=\pi_d\circ F_d^{-1}\circ \pi_{d+1}^{-1} (\varphi_d([P])).\]
Since $\pi_{d+1}^{-1}(\varphi_d([P]))=SO(d+1)(F_d(P))$, and since Proposition \ref{prop:polygon_orbit_embedding} says
\[F_d^{-1}(SO(d+1)(F_d(P)))=O(d)(P)\]
we have
\[\varphi_d^{-1}\circ\varphi_d([P])=\pi_d(O(d)(P)).\]
Thus by the partition $O(d)(P)=SO(d)(P)\cup SO(d)(T(P))$ in \eqref{eqn:partition2} we have
\[\varphi_d^{-1}\circ\varphi_d([P])=\pi_d(SO(d)(P)\cup SO(d)(T(P)))=\{[P], [T(P)]\}\]
where $T$ is any element of $SO(d)^-$. The result now follows from Proposition \ref{prop:equal_orbits} which implies $[P]=[T(P)]$ if and only if $\dim(P)<d$.
\end{proof}

We will use Proposition \ref{prop:injective} in the proof of Theorem \ref{thm:stable} to show that $\varphi_d$ is injective when $d\geq n$, but it tells us more. The proposition along with its proof tells us exactly how $\varphi_d$ is {\em not} injective when $d<n$, namely, that reflections of $d$-dimensional polygons in $\R^d$ become identified in $\R^{d+1}$. We will see this exemplified in Section \ref{sec:example}. For now, we are finally ready to state our main theorem, and its proof is easy since we have done all the heavy lifting.

\begin{theorem}\label{thm:stable}
Let $\ell\in\interior(C_n)$. Then the directed system
\[M_2(\ell)\xrightarrow{\varphi_2} M_3(\ell)\xrightarrow{\varphi_3} M_4(\ell)\xrightarrow{\varphi_4}\cdots \xrightarrow{\varphi_{d-1}} M_d(\ell) \xrightarrow{\varphi_d}\cdots\]
of moduli spaces of $\ell$-gons stabilizes at $d=n$, that is, $M_d(\ell)\approx M_n(\ell)$ if and only if $d\geq n$.
\end{theorem}

\begin{proof}
We show that for $\ell\in\interior(C_n)$ the map $\varphi_d:M_d(\ell)\to M_{d+1}(\ell)$ is a homeomorphism if and only if $d\geq n$. Let $\ell\in\interior(C_n)$. By Proposition \ref{prop:surjective} $\varphi_d$ is surjective if and only if $d\geq n-1$, so it remains to show that $\varphi_d$ is injective if and only if $d\geq n$. By Proposition \ref{prop:injective}, $\varphi_d$ is injective if and only if $M_d(\ell)$ does not contain $d$-dimensional polygons. By Proposition \ref{prop:dimensionful} this happens if and only if $d\geq n$.
\end{proof}

\section{Moduli spaces of equilateral $4$-gons}\label{sec:example}

We end with an example of a directed system of moduli spaces of $\ell$-gons, in the hope that it aids the reader's intuition for these systems. Let $\ell=(1,1,1,1)\in\interior(C_4)$. Theorem \ref{thm:stable} says that the directed system of moduli spaces of $\ell$-gons stabilizes at $d=4$, and thus we have
\[M_2(\ell)\xrightarrow{\varphi_2} M_3(\ell)\xrightarrow{\varphi_3} M_4(\ell) \approx M_5(\ell) \approx M_6(\ell) \approx \cdots.\]
Subsections \ref{subsec:R_2}, \ref{subsec:R_3}, and \ref{subsec:R_4} give descriptions of the moduli spaces $M_2(\ell), \, M_3(\ell)$, and the stable limit $M_4(\ell)$, respectively, and Figure \ref{fig:R_2_3_4} provides a picture of each. In the figure, several points in the moduli spaces are marked with their corresponding polygons drawn next to them. A given picture should not be thought of as a static image, but as flows of polygons smoothly transforming into one another (see \cite{millson2} for more on Hamiltonian flows in polygon spaces). Also, the pictures should not be thought of as wholly distinct from one another, but instead as being related to each other by the identification of some polygons and the arising of others at each step in the directed system.

\subsection{$M_2(\ell)$}\label{subsec:R_2}

The moduli space $M_2(\ell)$ of $\ell$-gons in $\R^2$ is homeomorphic to three circles, every two of which meet at a single point (\cite{millson}, p15). The points of intersection correspond to $1$-dimensional polygons. A given circle is partitioned by these points of intersection into two arcs of $2$-dimensional polygons, where the polygons along one arc are reflections of the polygons along the other arc. See the polygons next to the marked points in $M_2(\ell)$ in the top left image of Figure \ref{fig:R_2_3_4} and try to imagine the polygons that are not drawn. For example, the curves going from the square at the top of the picture to either endpoint of the arc it is sitting on, those consist of parallelograms that get narrower and narrower as you approach the $1$-dimensional degenerate parallelograms at the arc's endpoints.

\subsection{$M_3(\ell)$}\label{subsec:R_3}

The moduli space $M_3(\ell)$ is homeomorphic to a sphere. As we move from polygons in $\R^2$ to polygons in $\R^3$, two things happen. First, every $2$-dimensional polygon in $M_2(\ell)$ becomes identified with its reflection, and thus each pair of arcs comprising a circle in $M_2(\ell)$ is projected onto a single arc in $M_3(\ell)$; these three arcs make up the equator in $M_3(\ell)$. Second, $3$-dimensional polygons appear that did not exist in $M_2(\ell)$; they comprise the two hemispheres. Analogous to the arc pairs consisting of reflected polygons in $\R^2$, the hemisphere pairs consist of reflected polygons in $\R^3$---every polygon in the southern hemisphere has its reflection in the northern hemisphere.

\subsection{$M_d(\ell)$, $d\geq 4$}\label{subsec:R_4}

Next, we arrive at the moduli space $M_4(\ell)$ which is homeomorphic to a disc. As we move from $\R^3$ to $\R^4$, no new polygons appear because the maximum dimension of a $4$-gon is $3$---in other words $\varphi_3$ is surjective. However, $\varphi_3$ is not injective, since $3$-dimensional polygons and their reflections in $M_3(\ell)$ get identified in $M_4(\ell)$. Thus the northern and southern hemispheres in $M_3(\ell)$ collapse to form the interior of the disk. This disk is the stable limit. As we move forward to $M_5(\ell)$ no new polygons appear because they all appeared back in $M_3(\ell)$. Moreover, no new identifications are made: if $\varphi_4([P])=\varphi_4([Q])$ then $P$ and $Q$ are rotations or reflections of each other in $\R^4$, but since their dimension is less than $4$ every reflection is obtained by rotation, so $[P]=[Q]$.

%
%

\subsection{Questions}\label{subsec:questions}
We conclude with some questions inspired by this example. First, is the stable limit always contractible and connected? Does it always have boundary, and if so does the interior consist of polygons of maximum dimension? We also notice that the $1$-dimensional polygons partition $M_2(\ell)$ into connected components of $2$-dimensional polygons, and the $1$- and $2$-dimensional polygons partition $M_3(\ell)$ into connected components of $3$-dimensional polygons. Can this be generalized? The answers to all these questions may come from an algebro-geometric view of polygon spaces. It is not mentioned earlier, but $V_d(\ell)$ is a real algebraic variety, and thus $M_d(\ell)$ is real semi-algebraic space. Is there an analogous directed system of semi-algebraic spaces? Is there a correspondence between polygon dimension and the dimension of the (pieces of) subvarieties comprising $M_d(\ell)$?

\clearpage

\begin{figure}[h!]
\centering
\includestandalone[scale=.8]{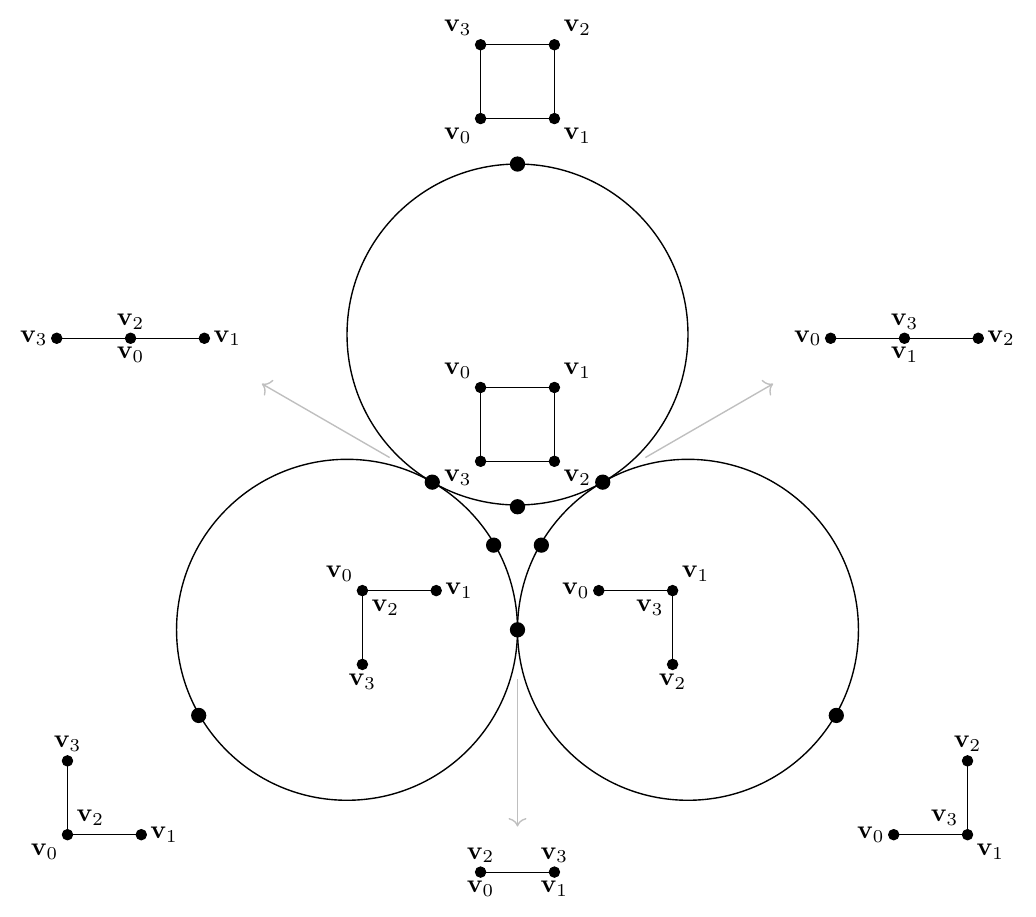}\qquad
\includestandalone[scale=.7]{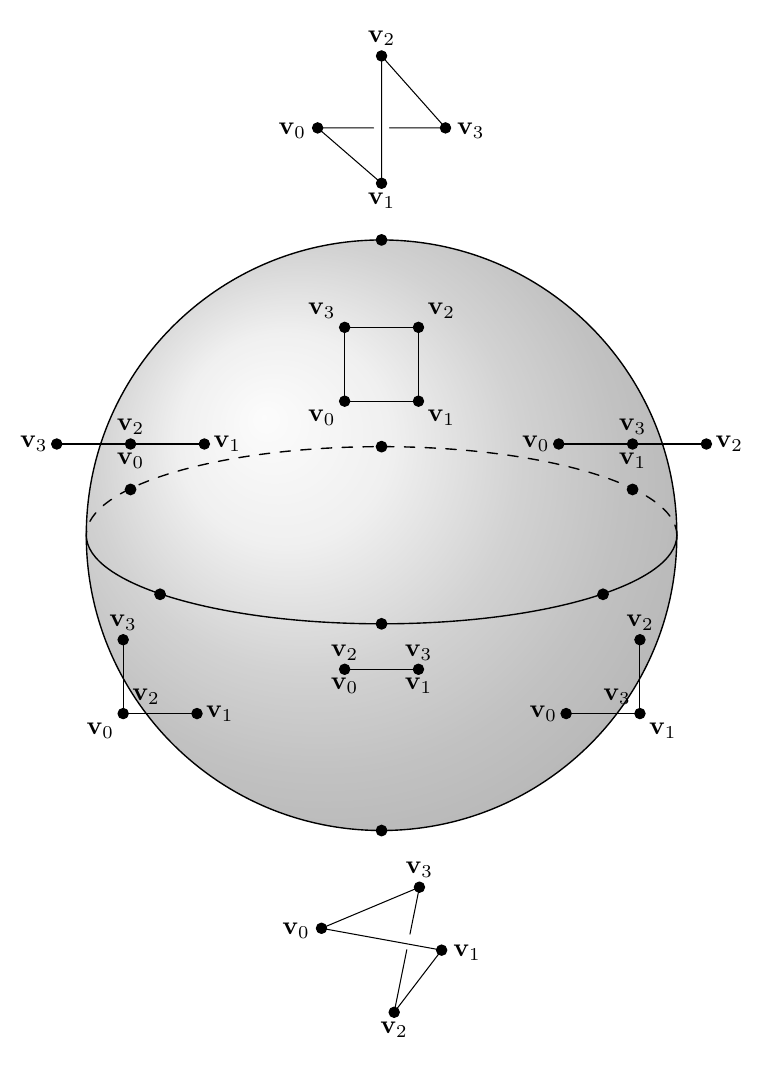}
\vspace{.2in}

\includestandalone[scale=.8]{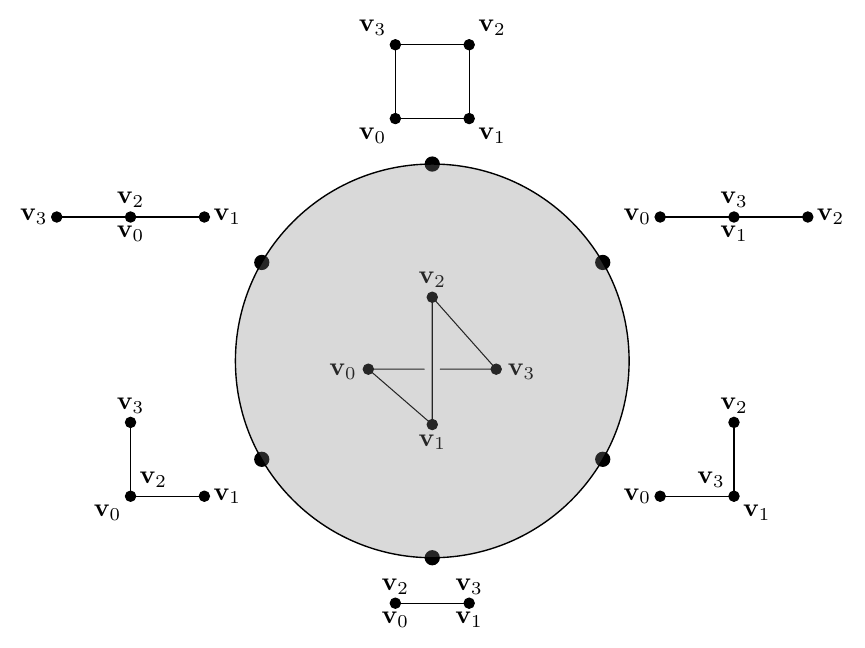}
\caption{Clockwise from top-left, the moduli spaces $M_2(\ell)$, $M_3(\ell)$, and $M_d(\ell), d\geq 4$, for $\ell=(1,1,1,1)$.} 
\label{fig:R_2_3_4}
\end{figure}


\bibliographystyle{alpha}
\bibliography{references}

\begin{thebibliography}{{Kou}14}

\bibitem[CS16]{Cantarella2016TheSG}
J.~Cantarella and C.~Shonkwiler.
\newblock The symplectic geometry of closed equilateral random walks in
  3-space.
\newblock {\em Annals of Applied Probability}, 26:549--596, 2016.

\bibitem[FF13]{farber}
Michael Farber and Viktor Fromm.
\newblock The topology of spaces of polygons.
\newblock {\em Trans. Amer. Math. Soc.}, 365(6):3097--3114, 2013.

\bibitem[GM13]{higgs}
Leonor Godinho and Alessia Mandini.
\newblock Hyperpolygon spaces and moduli spaces of parabolic {H}iggs bundles.
\newblock {\em Adv. Math.}, 244:465--532, 2013.

\bibitem[HK96]{Hausmann1996PolygonSA}
J.~Hausmann and A.~Knutson.
\newblock Polygon spaces and grassmannians.
\newblock 1996.

\bibitem[HK97]{grassmannians}
Jean-Claude Hausmann and Allen Knutson.
\newblock Polygon spaces and {G}rassmannians.
\newblock {\em Enseign. Math. (2)}, 43(1-2):173--198, 1997.

\bibitem[HMM11]{manon}
Benjamin Howard, Christopher Manon, and John Millson.
\newblock The toric geometry of triangulated polygons in {E}uclidean space.
\newblock {\em Canad. J. Math.}, 63(4):878--937, 2011.

\bibitem[Kin99]{signature}
Henry~C. King.
\newblock Planar linkages and algebraic sets.
\newblock In {\em Proceedings of 6th {G}\"{o}kova {G}eometry-{T}opology
  {C}onference}, volume~23, pages 33--56, 1999.

\bibitem[KM95]{millson}
Michael Kapovich and John Millson.
\newblock On the moduli space of polygons in the {E}uclidean plane.
\newblock {\em J. Differential Geom.}, 42(2):430--464, 1995.

\bibitem[KM96]{millson2}
Michael Kapovich and John~J. Millson.
\newblock The symplectic geometry of polygons in {E}uclidean space.
\newblock {\em J. Differential Geom.}, 44(3):479--513, 1996.

\bibitem[KM02]{history}
Michael Kapovich and John~J. Millson.
\newblock Universality theorems for configuration spaces of planar linkages.
\newblock {\em Topology}, 41(6):1051--1107, 2002.

\bibitem[{Kou}14]{kourganoff}
Micka{\"e}l {Kourganoff}.
\newblock {Universality theorems for linkages in the Minkowski plane}.
\newblock {\em arXiv e-prints}, page arXiv:1401.1050, Jan 2014.

\bibitem[Man09]{mandini}
Alessia Mandini.
\newblock The cobordism class of the moduli space of polygons in
  $\mathbb{R}^3$.
\newblock {\em J. Symplectic Geom.}, 7(1):1--27, 03 2009.

\end{thebibliography}

\Addresses

\end{document}